\documentclass[aap,preprint]{imsart}

\RequirePackage[OT1]{fontenc}
\RequirePackage{amsthm,amsmath,amsfonts,amssymb}
\RequirePackage[numbers]{natbib}
\RequirePackage[colorlinks,citecolor=blue,urlcolor=blue]{hyperref}

\doi{10.1214/16-AAP1228}
\pubyear{2017}
\volume{27}
\issue{2}
\firstpage{1171}
\lastpage{1189}

\startlocaldefs
\numberwithin{equation}{section}
\theoremstyle{plain}
\newtheorem{thm}{Theorem}[section]
\newtheorem{lemma}{Lemma}[section]
\newtheorem{remark}{Remark}[section]
\newtheorem{Corollary}{Corollary}[section]
\newtheorem{proposition}{Proposition}[section]
\def \E {{\mathbb E}}
\def \P {{\mathbb P}}
\def \RR {{\mathbb R}}

\def \lf {\underline{f}}
\def \uf {\bar{f}}
\endlocaldefs

\begin{document}

\begin{frontmatter}
\title{Logarithmic tails of sums of products of positive random variables bounded by one}
\runtitle{Logarithmic tails of sums of products}

\begin{aug}
\author{\fnms{Bartosz} \snm{Ko{\l}odziejek}\ead[label=e1]{kolodziejekb@mini.pw.edu.pl}}
\runauthor{B. Ko{\l}odziejek}

\affiliation{Warsaw University of Technology}

\address{Faculty of Mathematics and Information Science\\Warsaw University of Technology\\ Koszykowa 75\\00-662 Warsaw, Poland \\
\printead{e1}}

\end{aug}

\begin{abstract}
In this paper, we show under weak assumptions that for $R\stackrel{d}{=}1+M_1+M_1M_2+\ldots$, where $\P(M\in[0,1])=1$ and $M_i$ are independent copies of $M$, we have $\ln\P(R>x)\sim C\, x\ln\P(M>1-1/x)$ as $x\to\infty$. The constant $C$ is given explicitly and its value depends on the rate of convergence of $\ln\P(M>1-1/x)$. Random variable $R$ satisfies the stochastic equation $R\stackrel{d}{=}1+MR$ with $M$ and $R$ independent, thus this result fits into the study of tails of iterated random equations, or more specifically, perpetuities.\end{abstract}

\begin{keyword}[class=MSC]
\kwd[Primary ]{60H25}
\kwd[; secondary ]{60E99}
\end{keyword}

\begin{keyword}
\kwd{Random iterated equation}
\kwd{perpetuity}
\kwd{regular variation}
\kwd{tail asymptotic}
\end{keyword}

\end{frontmatter}

\section{Introduction}
In the present paper we consider a random variable $R$ given by the solution of the stochastic equation
\begin{align}\label{defperp}
R \stackrel{d}{=}Q + MR,\qquad R\mbox{ and }(Q,M)\mbox{ independent.}
\end{align}
When $R$ is the solution of \eqref{defperp}, then following a custom from insurance mathematics, we call $R$ a \emph{perpetuity}. In this scheme, $Q$ represents the payment, and $M$ the discount factor, both being subject to random fluctuations. Then $R$ is the present value of a commitment to pay the value of $Q$ every year in the future (see \eqref{perpd} below). Such stochastic equation appears in many areas of applied mathematics; for a broad list of references see, for example, \citep{DF99} and \citep{EG93}.

Under suitable assumptions [see \eqref{cond}] on $(Q,M)$, one can think of $R$ as a limit in distribution of the following iterative scheme:
\begin{align}\label{it}
R_n = Q_n + M_nR_{n-1}, \qquad n \geq 1,
\end{align}
where $R_{n-1}$ and $(Q_n,M_n)$ are independent, $R_0$ is arbitrary and $(Q_n,M_n)$, $n \geq 1$, are i.i.d. copies of $(Q,M)$. Writing out the above recurrence and renumbering
the random variables $(Q_n,M_n)$, we see that $R$ may also be defined by
\begin{align}\label{perpd}
R \stackrel{d}{=}\sum_{j=1}^\infty Q_j \prod_{k=1}^{j-1}M_k,
\end{align}
provided that the series above converges in distribution. Sufficient conditions for
the almost sure convergence of the series in \eqref{perpd} have been given by Kesten \citep{Kes73} who
also considered a multidimensional case when $M$ is a matrix and $Q$ is a vector. For a detailed
discussion of sufficient and necessary conditions in one-dimensional case, we refer to Vervaat \citep{ver79} and Goldie and Maller \citep{GM00}; conditions
\begin{align}\label{cond}
\E\ln^+ |Q| < \infty\qquad\mbox{ and }\qquad\E\ln|M| < 0
\end{align}
suffice for the almost sure convergence of the series in \eqref{perpd}.

The main focus of research is the tail behavior of $R$. A classical result of Kesten \citep{Kes73} (see also \citep{Gri75,Gol91} for one-dimensional case) states that if
there exists a constant $\kappa>0$ such that $\E|M|^\kappa=1$, $\E|M|^\kappa\ln^+|M|<\infty$ and $0<\E|Q|^\kappa<\infty$ (supplemented with some additional assumption of nondegeneracy and that the distribution of $\ln|M|$ given $M\neq 0$ is nonarithmetic), then there exists a positive constant $C$ such that
$$\P(|R| > x) \sim Cx^{-\kappa}.$$
Throughout the paper, the symbol $f(x)\sim g(x)$ means that $f(x)/g(x)\to1$ as $x\to\infty$.  For representations and bounds for $C$, see \citep{Gol91, ESZ09, BDZ, DLNT}.
Note that such $\kappa$ exists only if $\P(|M|>1)>0$.
The complementary case, $\P(|M|\leq1)= 1$, is not so well understood and the exact asymptotic of $\ln\P(R>x)$ is known in special cases only. The results in this area with references are given in Section \ref{known}.

In the present paper, we find the asymptotic behaviour of $\ln\P(R>x)$, when $Q$ is degenerate and $\P(0\leq M\leq1)=1$ with some additional weak assumptions on the tail of $M$. Under such conditions, we were able to link the behaviour of $M$ near $1$ with the log-tail of $R$. In particular, we show that if $Q=1$ with probability $1$ and $\ln\P(M>1-1/x)\sim-cx^{\alpha-1}$ for some $\alpha>1$ and $c>0$, then $R$ exhibits Weibull-like tails: 
$$\ln\P(R>x)\sim -c\beta^{\alpha-1} x^\alpha,$$ 
where $\beta>1$ is such that $\alpha^{-1}+\beta^{-1}=1$. We also cover the cases when $x\mapsto -\ln\P(M>1-1/x)$ is slowly varying or belongs to the class $\Gamma$ of rapidly varying functions. 

In the case of $Q$ degenerate, finding the asymptotic of the tail $\P(R>x)$ is a very hard task, which, at the moment, has been solved only when $M^\alpha$ is uniformly distributed, $\alpha>0$. 
However, the knowledge of asymptotics of log-tail of $R$ is now fairly complete. At last, we would like to stress that systematic study of the tail of $R$ under the assumption that $|M|\leq 1$ with probability $1$ is not as limiting as it would appear at the first glance and this is due to following observation from \citep{ver79}.
\begin{thm}
Let $R$ be a solution of \eqref{defperp} for some $(Q,M)$, such that $\P(|M|\neq 1)>0$ and $\P(M=0)=0$. Then for all $\varepsilon>0$ there exists a random pair $(Q',M')$ such that $\P(0<M'<\varepsilon)=1$ and $R$ is a solution of \eqref{defperp} with $(Q',M')$ instead of $(Q,M)$.
\end{thm}

The paper is organized as follows. We start in the next section with basic definitions and theorems
regarding regular and rapid variation. In Section \ref{known}, we briefly recall some known results regarding the asymptotic behaviour of the tail of $R$ in the case $\P(|M|\leq1)=1$. Statement of the main result is given in Section \ref{secmain}. For suitably chosen $f$, we independently prove the upper bound $$\limsup_{x\to\infty}\frac{\ln\P(R>x)}{f(x)}\leq B$$ and the lower bound $$\liminf_{x\to\infty}\frac{\ln\P(R>x)}{f(x)}\geq B$$ for some constant $B$. The proof of the upper bound is based on an inductive argument given in \citep{HitWes09} (Section \ref{secup}), while for the lower bound we improve the Goldie--Gr{\"u}bel inequality (Section \ref{lower}).
Section \ref{remarks} is devoted to the study of negative $M$ and the case not covered by Theorem \ref{mainr}, that is, when $x\mapsto -\ln\P(M>1-1/x)$ is rapidly varying, but does not belong to the class $\Gamma$.

\section{Regular variation}\label{Preli}
In this section we give a brief introduction to the theory of regular and rapid variation. For further details we refer to Bingham et al. \citep{BGT89}.

A measurable function $f\colon(0,\infty)\to(0,\infty)$ is called \emph{regularly varying with index $\rho$} (denoted $f\in R(\rho)$), $|\rho|<\infty$, if for all $\lambda>0$,
\begin{align}\label{reg}
\lim_{x\to\infty}\frac{f(\lambda x)}{f(x)}=\lambda^\rho.
\end{align}
The convergence above is locally uniform (\cite[Theorem 1.5.2]{BGT89}).
If $f\in R(0)$, then $f$ is called a \emph{slowly varying function}. The class of slowly varying functions is a fundamental part of the Karamata's theory of regular variability, since if $f\in R(\rho)$, then $f(x)=x^\rho L(x)$, where $L\in R(0)$. 

We say that a positive function varies smoothly with index $\rho$ ($f\in SR(\rho)$), if $f\in C^\infty$ and for all $n\in\mathbb{N}$,
$$\lim_{x\to\infty}\frac{x^n f^{(n)}(x)}{f(x)}=\rho(\rho-1)\ldots(\rho-n+1).$$
In particular, if $f\in SR(\rho)$ then $x^2 f''(x)/f(x)\sim \rho(\rho-1)$, hence for $\rho>1$ second derivative $f''(x)$ is positive for large $x$ so $f$ is ultimately strictly convex (ultimately here and later means ``for large values of the argument'').
By the Smooth Variation Theorem for any $f\in R(\rho)$, there exist $\lf, \uf\in SR(\rho)\subset C^\infty$ with $\lf(x)\sim \uf(x)$ and $\lf\leq f\leq \uf$ on a neighbourhood of infinity.

If $f\in R(\rho)$, $\rho>0$, then for any $A>1$ and $\delta\in(0,\rho)$, there exist $X=X(A,\delta)$ such that (Potter's theorem)
$$\frac{1}{A}\left( \frac{y}{x}\right)^{\rho-\delta}\leq \frac{f(y)}{f(x)}\leq A \left( \frac{y}{x}\right)^{\rho+\delta},\qquad y>x>X.$$

Potter's theorem implies following easy lemma, which will be very useful later on.
\begin{lemma}\label{lem1}
If $f\in R(\rho)$, $\rho> 0$, $g(x)\to\infty$ as $x\to\infty$ and $$\lim_{x\to\infty}\frac{f(x)}{f(g(x))}=L,$$ then 
$\lim_{x\to\infty}x/g(x)=L^{1/\rho}.$
\end{lemma}
\begin{proof}
First observe that Potter bounds imply the existence of constants $\lambda_1, \lambda_2>0$ such that for $x$ large enough $\lambda_1\leq \frac{x}{g(x)} \leq \lambda_2$. Let $x_n\to\infty$.
From any sequence $y_n=x_n/g(x_n)$, one may select a convergent subsequence $(y_{n_k})_k$, that is, $y_{n_k}=x_{n_k}/g(x_{n_k})$ converges to $\lambda\in[\lambda_{1},\lambda_2]$ say.
Since the convergence in \eqref{reg} is uniform, one has
$f(x_{n_k})/f(g(x_{n_k}))=f(y_{n_k} g(x_{n_k}))/f(g(x_{n_k}))\to \lambda^\rho.$
Thus, $\lambda^\rho=L$, and so $x/g(x)$ converges to $\lambda=L^{1/\rho}$.
\end{proof}

If $f\colon(0,\infty)\to(0,\infty)$ is measurable and
\begin{align*}
\lim_{x\to\infty}\frac{f(\lambda x)}{f(x)}=\begin{cases}\infty, & \lambda>1 \\ 1, & \lambda=1, \\ 0, & 0<\lambda<1, \end{cases}
\end{align*}
we call $f$ \emph{rapidly varying} (denoted $f\in R(\infty)$). The rapidly varying functions, however, in general do not possess properties that we need in the proof, therefore we will restrict our considerations to a subclass of $R(\infty)$ called $\Gamma$. This class appears in a natural way when dealing with convergence in distribution of sequences of partial maxima of independent, identically distributed random variables: if $X_1, X_2,\ldots$ are independent random variables with a common distribution $F$ (assume for simplicity $F(x)<1$ for all $x$), then there exist sequences of real constants $(a_n)$ and $(b_n)$ such that
$$\P(a_n(\max_{k\leq n}X_k-b_n)\leq x)\to\exp(-e^{-x}),\qquad\mbox{ as }n\to\infty$$
if and only if the function $U(x)=1/(1-F(x))$ is in $\Gamma$ (Gnedenko \citep{Gne43}).

The class $\Gamma$ consists of nondecreasing and right-continuous functions $f$ for which there exists a measurable function $g\colon \RR\to(0,\infty)$ such that (see \cite[Section 3.10]{BGT89})
\begin{align}\label{Gamma}
\lim_{x\to\infty}\frac{f\left(x+ug(x)\right)}{f(x)}= e^u,\quad u\in\mathbb{R}.
\end{align}
It can be shown that for $f\in\Gamma$ and any $\lambda>1$ we have $f(\lambda x)/f(x)\stackrel{x\to\infty}{\longrightarrow}\infty$, thus $\Gamma\subset R(\infty)$ except that the domain of definition of functions in $\Gamma$ is $\RR$, while those of $R(\infty)$ are $(0,\infty)$.

Function $g$ in \eqref{Gamma} is called \emph{an auxiliary function} and if $f$ has nondecreasing positive derivative, then one may take $g=f/f'$. In such case $f'$ also belongs to $\Gamma$ with the same auxiliary function as $f$. Every $f\in\Gamma$ has a representation (\citep{BH72})
$$f(x)=\exp\left\{ \eta(x)+\int_0^x \frac{1}{b(t)}dt\right\},$$
where $\eta$ is measurable, $b$ is a positive differentiable function such that $\eta(x)\to d$ and $b'(x)\to 0$ as $x\to\infty$.
Obviously, 
$$h(x):=\exp\left\{ d+\int_0^x \frac{1}{b(t)}dt\right\}\sim f(x).$$
Moreover, $h$ is at least twice differentiable and $h''(x)=h(x)(1-b'(x))/b(x)^2$, thus $h$ is ultimately strictly convex. 

The class $\Gamma$ is very rich: If $f_1\in R(\rho)$, $\rho>0$ and $f_2\in\Gamma$, then $f_1\circ f_2\in\Gamma$ (\cite[Proposition 3.10.12]{BGT89}). The same holds if $f_1\in\Gamma$ and $f_2'\in R(\rho)$ with $\rho>-1$ or if $f_1, f_2'\in\Gamma$ (\cite[p.191]{BGT89}). An easy example of a function from $R(\infty)\setminus\Gamma$ is $f(x)=\exp(x-\cos x)$.

If $f\in\Gamma$ with auxiliary function $g$, then $f^{-1}$ (generalized inverse) is slowly varying and for $\lambda>0$,
\begin{align}\label{gg}
\lim_{x\to\infty}\frac{f^{-1}(\lambda x)-f^{-1}(x)}{g(f^{-1}(x))}=\ln\lambda.\end{align}

\section{Previous results}\label{known}
In the case of bounded $Q$ and $\P(|M|\leq1)=1$, the knowledge of asymptotic behaviour of the tail $\P(R>x)$ or log-tail $\ln\P(R>x)$ is very scarce. The fact that there are few examples of explicit solutions of \eqref{defperp} certainly does not help.

Asymptotic of the tail $\P(R>x)$ for the case $M=U^{1/\alpha}$, where $U\sim \mathrm{U}[0,1]$ and $\alpha>0$, was treated in \citep{ver72} with the use of asymptotic result from \citep{bru51}. If $\alpha=1$, then $\P(R-1>x)\sim \varrho(x)$, where $\varrho$ is the Dickman-de Bruijn function, which appears, for example, in the number theory.

Some more recent results appeared in \citep{GG96}, where the authors considered nondegenerate $Q$ and $M$ such that
$$\P(M=0)>0\quad\mbox{ or }\quad\E\ln^+ |Q|<\infty.$$
The distribution of $M$ is said to be equivalent at $1$ to uniform distribution $\mathrm{U}([0,1])$ if
for every $\varepsilon>0$ there exist positive constants $c$ and $C$ such that 
\begin{align}\label{eqiv}
c\overline{F}_{\mathrm{U}([0,1])}(1-\delta)\leq \overline{F}_M(1-\delta)\leq C\overline{F}_{\mathrm{U}([0,1])}(1-\delta),\qquad \forall\,\delta\in(0,\varepsilon].
\end{align}
Here and henceforth $\overline{F}$ denotes the tail function $\overline{F}=1-F$. If the distribution of $M$ is equivalent at $1$ to $\mathrm{U}([0,1])$
and $Q$ is such that for $q_+=\sup\{ q\colon \P(Q> q)>0\}$, one has $0<q_+\leq\infty$, then it was shown that 
\begin{align*}
\lim_{x\to\infty}\frac{\ln \P(R>x)}{x\ln x}=-\frac{1}{q_+}.
\end{align*}
When $Q$ is degenerate, this result follows from \citep{ver72}.

The most recent results come from \citep{HitWes09}, where three different distributions of $M$ along with degenerate and positive $Q$ were considered.
For instance, the authors showed that if $M$ has a distribution equivalent at $1$ (in the sense of \eqref{eqiv}) to $F(x)=1-\exp\{-\beta(-\ln(1-x))^\gamma\}$, for $\beta,\eta>0$, then 
$$\lim_{x\to\infty}\frac{\ln\P(R>x)}{\beta\frac{x}{q}(\ln x)^\eta}\to -1.$$
This is an example when $f(x):=-x\ln\P(M>1-1/x)=\beta x(\ln x)^\eta\in R(1)$.

Their next example is $f\in R(r)$, $r>1$. They showed that if the distribution of $M$ is equivalent at $1$ to a distribution with given density, then
$$-\infty<c_1\leq \liminf_{x\to\infty} \frac{\ln\P(R>x)}{x^r}\leq \limsup_{x\to\infty} \frac{\ln\P(R>x)}{x^r}\leq c_2<0,$$
but $c_2>c_1$ and this does not imply that the limit of $\ln\P(R>x)/x^r$ even exists. Their upper bound however is optimal, but this will follow from the present paper.

In the third example (this time $f$ comes from the class $\Gamma$) they showed that there exist a distribution of $M$ such that
$$\forall\,B>q\quad\limsup_{x\to\infty}\frac{\ln\P(R>x)}{B\exp(x/B)}\leq -\frac{1}{e}$$
and
$$\forall\,B<q\quad\liminf_{x\to\infty}\frac{\ln\P(R>x)}{B\exp(x/B)}\geq \frac{\ln(1-B/q)}{B}.$$
This result is even less satisfactory than the one from the previous example, because here one may not take the same comparison functions (setting $B=q$ one obtains $-\infty\leq \liminf\leq\limsup\leq -1/e$). 

Some general bounds for arbitrary distribution of $M$ can be obtained from the following result due to Hitczenko \citep{Hit10}: if $Q=q>0$ and $\P(0\leq M\leq 1)=1$, then for sufficiently large $x$,
\begin{align}\label{genhit}
2\ln 2\frac{x}{q}\ln \P\left(M>1-\frac{q}{2x}\right)\leq \ln\P(R>x)\leq 4\frac{x}{q}\ln\P\left(M>1-\frac{2q}{x}\right).
\end{align}
Heuristically, one would expect that $\ln\P(R>x)\sim c x/q\ln\P(1-q/x)$ and this actually happens to be the proper estimate, as it will be shown below.

The case of bounded $Q$ and $\P(|M|\leq 1-\varepsilon)=1$ for $\varepsilon>0$ is uninteresting, because then $\P(R>x)=0$ for $x$ large enough. In order to exclude this case from further considerations we will assume that 
\begin{align}\label{msup}\sup\{x\colon \P(M>x)>0\}=1.\end{align}

If $Q$ is not bounded, the relation between the tail of $R$ and the tail of $M$ in \eqref{genhit} is not always the case. Grincevi\'{c}ius \citep{Gri75} has shown (this result was improved and the proof corrected in \citep{Gre94}) that if $M$ is positive with probability $1$, $\E M^\alpha<1$, $\E M^{\alpha+\varepsilon}<\infty$ and $\P(Q>x)\in R(-\alpha)$ for $\alpha, \varepsilon>0$, then
$$\P(R>x)\sim \frac{1}{1-\E M^\alpha}\P(Q>x).$$
The case when $x\mapsto\P(Q>x)\in R(0)$ was treated recently in \citep{Dys15}.

The tail and log-tail asymptotics of a perpetuity $R=\int_{0}^\infty \exp(-X(s))ds$, where $X$ is a L\'evy process, was studied in \citep{MZ06,Riv12}.

Some bounds for the density as well as for the tail functions of $R$ were also obtained in \citep{KN08}.

\section{Main result}\label{secmain}
\begin{thm}\label{mainr}
Let $R$ satisfy \eqref{defperp} with $\P(M\in[0,1))=1$. Assume that \eqref{msup} holds and $Q=q>0$ with probability $1$.
Define $f(x)=-x\ln\P(M>1-1/x)$, $x\geq 1$.

If $f\in R(r^\ast)$, $r^\ast> 1$, then
\begin{align*}
\lim_{x\to\infty}\frac{\ln \P(R>x)}{f(\frac x{q})}=-r^{r^\ast-1},
\end{align*}
where $1/r+1/r^\ast=1$.

If $f\in \Gamma\subset R(\infty)$, then
\begin{align*}
\lim_{x\to\infty}\frac{\ln \P(R>x)}{f(\frac x{q})}=-e.
\end{align*}

If $f\in R(1)$ is ultimately strictly convex, then
\begin{align*}
\lim_{x\to\infty}\frac{\ln \P(R>x)}{f(\frac x{q})}=-1.
\end{align*}
\end{thm}
\begin{remark}
Under the assumptions of Theorem \eqref{mainr}, the series in \eqref{perpd} converges and the perpetuity $R$ exists.
\end{remark}
\begin{remark}
Note that if $1/r+1/r^\ast=1$, then
$$\lim_{r^\ast\to1^+}r^{r^\ast-1}=1\quad\mbox{ and }\quad\lim_{r^\ast\to\infty}r^{r^\ast-1}=e,$$
therefore we have a kind of continuity of the constant here.
\end{remark}

Also, we have the following result, which complements Theorem \ref{mainr}.
\begin{Corollary}\label{ato}
Let $R$ satisfy \eqref{defperp} with $\P(M\in[0,1])=1$ and $\P(M=1)\in(0,1)$. If $Q=q>0$ with probability $1$, then
\begin{align*}
\lim_{x\to\infty}\frac{\ln \P(R>x)}{x}=\frac{1}{q}\ln\P(M=1).
\end{align*}
\end{Corollary}
\begin{proof}
From \citep[Theorem 1.7]{AIR09} specialized to our case, it follows that $\E e^{\lambda R}$ exists if $\lambda< -\ln \P(M=1)/q$. 
By Markov's inequality we have for $\lambda>0$.
$$\P(R>x)\leq e^{-\lambda x} \E e^{\lambda R}.$$
Taking the logarithms of both sides, dividing by $x$, taking $\limsup_{x\to\infty}$ and passing with $\lambda$ to $-\ln \P(M=1)/q$, we obtain
$$\limsup_{x\to\infty} \frac{\ln\P(R>x)}{x}\leq \frac 1q\ln\P(M=1).$$

Recall that $R\stackrel{d}{=}q\sum_{j=1}^\infty \prod_{k=1}^{j-1} M_k$ and let $R_n^\ast:=q\sum_{j=1}^{n+1} \prod_{k=1}^{j-1} M_k$. Since $\P(R_n^\ast=q(n+1))=\P(M=1)^n$, we have
\begin{align*}
\P(R> x)\geq \P(R_{\left\lceil \frac{x}{q}\right\rceil}^\ast> x)\geq \P(M=1)^{\left\lceil \frac{x}{q}\right\rceil},
\end{align*}
therefore,
$$\liminf_{x\to\infty} \frac{\ln\P(R>x)}{x}\geq \frac1q\ln\P(M=1).$$
\end{proof}
Note that an upper bound under the assumptions of Corollary \ref{ato} was considered in \citep[Corollary 2.2]{GG96}, but there it was shown there only that 
$$\limsup_{x\to\infty} \frac{\ln\P(R>x)}{x} \leq \frac1q\ln \E M.$$

\section{Upper bound}\label{secup}
In this section, we will prove that $$\limsup_{x\to\infty}\frac{\ln\P(R>x)}{-\frac xq\ln\P(M>1-\frac qx)}\leq -B$$ for some constant $B$.
We start with the following easy result with large deviations' flavour, which can be proved quickly by Markov's inequality.
\begin{proposition}\label{mark}
Suppose that 
$$\E e^{zX}\leq e^{f(z)}, \qquad z>0,$$
for some function $f\colon\RR_+\to\RR_+$. Then
$$\ln\P(X>x)\leq -f^\ast(x),$$
where 
\begin{align}\label{conv}
f^\ast(x)=\sup\{zx-f(z)\colon z>0\}.
\end{align}
\end{proposition}

The function $f^\ast$ is called the \emph{convex conjugate} of $f$. By the very definition $f^\ast$ is convex and $f(x)+f^\ast(z)\geq xz$ for any $x,z>0$ (Young's inequality). Note that Young's inequality implies that $f^\ast(x)/x\to\infty$ as $x\to\infty$. If $f$ is convex and lower-semicontinuous, then $f^{\ast\ast}=f$ (\citep{Roc70}). 
Convex-conjugacy is order-reversing, that is, if $f\leq g$, then $f^\ast\geq g^\ast$.

If $f$ is differentiable and strictly convex, then the supremum \eqref{conv} is attained at $z=(f')^{-1}(x)$ and thus 
$f^\ast(x)=x(f')^{-1}(x)-f((f')^{-1}(x)).$
Moreover, $f'\circ (f^\ast)'=(f^\ast)'\circ f'=\mathrm{Id}$. Thus,
\begin{align}\label{eqstar}
f^\ast(x)=x(f^\ast)'(x)-f((f^\ast)'(x)).
\end{align}

We begin with a proof of the upper bound.
Let 
	$$f(x):=\begin{cases}-x\ln\P(M>1-\frac1x), & x\geq1 \\ 0, & x<1.\end{cases}$$
The function $f$ is right-continuous and strictly increasing. Without loss of generality, we may assume that $Q=1$ with probability $1$, that is, $q=1$.

Hitczenko and Weso{\l}owski \citep{HitWes09}, making use of \citep{GG96}, have developed a method for obtaining an upper bound for moment generating function of $R$, which then through Proposition \ref{mark} gives an upper bound for the tail of $R$. 
However, there is an important nuance in their inductive argument.
First, they show that if for some function $\psi$,
\begin{align}\label{HW}I_\psi(z):=e^z\E e^{\psi(zM)-\psi(z)}\leq 1\end{align}
for all $z>0$, then the inductive assumption $\E e^{z R_n}\leq e^{\psi(z)}$ for all $z>0$ holds for every $n$.
Indeed,
$$\E e^{z R_{n+1}}=e^{z}\E e^{z M_{n+1} R_{n}}\leq e^{z} \E e^{\psi(zM)}\leq e^{\psi(z)}.$$
There is no problem with starting the induction, since $R_0$ may be taken arbitrary.
The authors of \citep{HitWes09} later state that it is enough to assume that \eqref{HW} holds for large values of $z$ only ($z>N_0$, say). If it is so, we are able only to show that $\E \exp(zR_1)\leq \exp(\psi(z))$ for $z>N_1>N_0$. This then implies that $\E \exp(z R_{n})\leq \exp(\psi(z))$ holds for large $z$ and the lower bound for such $z$'s may depend on $n$. Therefore, letting $n\to\infty$ does not justify the upper bound for $\E \exp(z R)$.
Note that if \eqref{HW} holds for large $z$, this does not imply that it holds for all $z>0$. 

Luckily, there is a way out of this situation. Assume that $\P(M=1)=0$ (only such situation will be considered using the following scheme). We will show that there exists a constant $D$ such that $\E \exp(z R)\leq \exp(\psi(z)+D)$ for large $z$.
Let $N>0$ be such that \eqref{HW} holds for $z>N$. Define $\widetilde{\psi}(x)=a x-D$ for $x\in[0,N]$ and $\widetilde{\psi}(x)=\psi(x)$ for $x>N$. 
If $a>1/(1-\E M)$, then $I_{\widetilde{\psi}}'(0)=1-a+a\E M<0$, thus there exists $\varepsilon>0$ such that $I_{\widetilde{\psi}}(z)\leq 1$ for $z\in[0,\varepsilon)$.
For $z\in[\varepsilon,N]$ we have
$$I_{\widetilde{\psi}}(z)\leq e^N \E e^{-a \varepsilon(1-M)}$$
and the right-hand side tends to $0$ as $a\to\infty$, thus for $a$ large enough one has $I_{\widetilde{\psi}}(z)\leq 1$ for all $z\in[0,N]$.
On the other hand, for $z>N$ one has
$$I_{\widetilde{\psi}}(z)=I_\psi(z)+ e^z \E\left(e^{a z M-\psi(z)-D}-e^{\psi(z M)-\psi(z)}\right)I(zM\leq N),$$
and one may choose $D$ in such a way that $a x-D<\psi(x)$ for any $x\in[0,N]$, so that the second term above is nonpositive. 
Thus, $I_{\widetilde{\psi}}(z)\leq 1$ for all $z>0$. The inductive argument holds for $\widetilde{\psi}$, and finally we obtain the desired bound
$$\E e^{z R}\leq e^{\psi(z)+D}$$
for large $z$ ($z>N$). Note that as $x\to\infty$ the supremum in \eqref{conv} is attained at $z\to\infty$, thus for large $x$ one has $\widetilde{\psi}^\ast(x)=(\psi+D)^\ast(x)=\psi^\ast(x)-D$. Proposition \ref{mark} then implies that
\begin{align}\label{eq3}
\limsup_{x\to\infty}\frac{\ln \P(R> x)}{\psi^\ast(x)}\leq -1.
\end{align}
To obtain this result, we will actually show much more:
$$\E e^{z+\psi(zM)-\psi(z)}\to 0,\quad\mbox{ as }z\to\infty$$
and this will be done using asymptotic properties of $\psi$ only.

The proof is divided into three cases:
\begin{description}
\item[Case I] $f\in R(r^\ast)$, $r^\ast>1$,
\item[Case II] $f\in \Gamma$,
\item[Case III] $f\in R(1)$.
\end{description}

\subsection*{Case I}
We will show that for some $B>0$,
\begin{align}\label{eq1}
I(z):=\E \exp\{\lf_B^\ast(zM)-\lf_B^\ast(z)+z\}
\end{align}
converges to $0$ as $z\to\infty$, where $\lf\in SR(r^\ast)$, as in the Smooth Variation Theorem, is such that $\lf(x)\sim f(x)$, $\lf\leq f$  and $\lf_B^\ast(x):=B \lf^\ast(x/B)$. Note that $\lf_B^\ast$ may be chosen to be nondecreasing.

Let $F_M$ denote the cumulative distribution function of $M$. For $\varepsilon_1\in(0,1)$, we have
\begin{multline*}
I(z)
\leq \exp\left\{\lf_B^\ast(z(1-\varepsilon_1))-\lf_B^\ast(z)+z\right\}F_M(1-\varepsilon_1)\\-\int_{1-\varepsilon_1}^1 \exp\{\lf_B^\ast(zx)-\lf_B^\ast(z)+z\}\,d\overline{F}_M(x).
\end{multline*}

Since $\lf\in SR(r^\ast)$ we have $\lf'\in SR(r^\ast-1)$, thus $(\lf^\ast)'\in SR(1/(r^\ast-1))$ and $\lf^\ast\in SR(r)$, where $1/r+1/r^\ast=1$ (see also \citep[Theorem 1.8.10]{BGT89}).

The first term converges to zero, since
\begin{align*}
\lf_B^\ast(z(1-\varepsilon_1))-\lf_B^\ast(z)+z=\lf_B^\ast(z)\left( \frac{z}{\lf_B^\ast(z)}+\frac{\lf_B^\ast(z(1-\varepsilon_1))}{\lf_B^\ast(z)}-1 \right)\to-\infty,
\end{align*}
and $\lim_{z\to\infty} \lf_B^\ast(z(1-\varepsilon_1))/\lf_B^\ast(z)=(1-\varepsilon_1)^{r}<1$. 
After integrating by parts (the integrand is continuous) and changing the variable $x\mapsto 1-t$, we obtain
\begin{align*}
I(z)\leq z(\lf_B^\ast)'(z)\int_0^{\varepsilon_1} \overline{F}_M(1-t) e^{\lf_B^\ast(z(1-t))-\lf_B^\ast(z)+z}\,dt+o(1).
\end{align*}
$\lf^\ast$ is ultimately convex, so, for $z$ large enough we have
$$\lf_B^\ast(z(1-t))-\lf_B^\ast(z)\leq-zt(\lf_B^\ast)'(z(1-t))\leq -zt(\lf_B^\ast)'(z(1-\varepsilon_1)).$$ 
Recall that $\overline{F}_M(1-t)=\exp\{-t f(\frac1t)\}\leq\exp\{-t \lf(\frac1t)\}$, thus
\begin{align}\label{eq4}\begin{split}
I(z)\leq \varepsilon_1 z&(\lf_B^\ast)'(z)\exp\left\{z \vphantom{\frac12} \right. \\
&\left.-\inf_{t\in(0,\varepsilon_1)}\left\{t\lf\left(\frac1t\right)+zt(\lf_B^\ast)'(z(1-\varepsilon_1))\right\}\right\}+o(1).
\end{split}\end{align}
The infimum is attained at the point $t_0=t_0(z)$ such that
\begin{align}\label{inf}
\frac1{t_0}\lf'\left(\frac1{t_0}\right)-\lf\left(\frac1{t_0}\right)=z(\lf_B^\ast)'(z(1-\varepsilon_1)).
\end{align}
Set $\lf'(1/t_0)=x$, that is, $1/t_0=(\lf^\ast)'(x)$. Substituting these into the left-hand side of \eqref{inf} shows that it equals 
$(\lf^\ast)'(x)x-\lf\left((\lf^\ast)'(x)\right)$, which, by \eqref{eqstar}, equals $\lf^\ast(x)$. Thus, \eqref{inf} reads 
\begin{align}\label{inf2}
\lf^\ast(x)=z(\lf^\ast)'\left(\frac{1-\varepsilon_1}{B}z\right).
\end{align}
For $z\to\infty$, we have $x=x(z)\to\infty$. Moreover, 
\begin{align}\label{zt}t_0\lf\left(\frac1{t_0}\right)+zt_0(\lf_B^\ast)'(z(1-\varepsilon_1))=x.\end{align}
Thus, \eqref{eq4} simplifies to
\begin{align}\label{eq4p}
I(z)\leq \varepsilon_1 z(\lf_B^\ast)'(z)e^{z-x}+o(1).
\end{align}

Since $\lf^\ast\in SR(r)$, $r>1$, we know that
$z(\lf^\ast)'(z)\sim r\lf^\ast(z)$. 
Hence
\begin{align*}\lf^\ast(x)=z(\lf^\ast)'\left(\frac{1-\varepsilon_1}{B}z\right)\sim r\left(\frac{1-\varepsilon_1}{B}\right)^{r-1} \lf^\ast(z)\end{align*}
and Lemma \eqref{lem1} implies that
\begin{align}\label{Br}
\frac{x}{z}\to \left(r\left(\frac{1-\varepsilon_1}{B}\right)^{r-1}\right)^{1/r}.
\end{align}
Combining \eqref{eq4p} and \eqref{Br}, we obtain
\begin{align*}
I(z)&\leq \varepsilon_1 z(\lf_B^\ast)'(z) \exp\left\{z\left(1-\frac{x}{z}\right)\right\}+o(1) \\
& =\varepsilon_1 z(\lf_B^\ast)'(z) \exp\left\{z\left(1- r^{1/r}\left(\frac{1-\varepsilon_1}{B}\right)^{1/r^\ast} + o(1) \right)\right\}+o(1),
\end{align*}
thus if $1- r^{1/r}((1-\varepsilon_1)/B)^{1/r^\ast}<0$, then $I(z)\to0$ as $z\to\infty$. This is true if $B<(1-\varepsilon_1)r^{r^\ast-1}$.
Passing to the limit as $\varepsilon_1\to0$ and $B\to {r^{r^\ast-1}}$, by \eqref{eq3} we obtain
\begin{align*}
\limsup_{x\to\infty}\frac{\ln \P(R>x)}{f(x)}=\limsup_{x\to\infty}\frac{\ln \P(R>x)}{\lf^{\ast\ast}(x)}\leq-r^{r^\ast-1}.
\end{align*}
The above equality follows from the fact that $\lf$ is ultimately strictly convex and thus $$\lf_B^{\ast\ast}(x)=B\lf(x)\sim B f(x).$$

\subsection*{Case II}
For $\varepsilon>0$, define $\lf(x)=e^{-\varepsilon}h(x)$, where $h(x)$ is defined as in Section \ref{Preli}. We have $h(x)\sim f(x)$, $h(x)$ is ultimately strictly convex and $\lf(x)\leq f(x)$ for large values of $x$.
Since $\lf'\in\Gamma$, we know that $(\lf')^{-1}=(\lf^\ast)'$ is slowly varying, and thus $\lf^\ast\in R(1)$.
Following the same approach as in Case I, we obtain (see \eqref{Br}, here $r=1$)
$$
\frac{x}{z}\to 1,
$$
thus in this case we need a more sophisticated approach.

Again using the fact that $(\lf')^{-1}=(\lf^\ast)'$, by \eqref{gg} we obtain
\begin{align}\label{psipl}
\lim_{x\to\infty}\frac{(\lf^\ast)'(\lambda x)-(\lf^\ast)'(x)}{g((\lf^\ast)'(x))}=\ln\lambda.
\end{align}
Moreover, $g=\lf/\lf'$, thus $g((\lf^\ast)'(x))=\lf((\lf^\ast)'(x))/x$

Recall that $\lf^\ast(x)=x(\lf^\ast)'(x)-\lf((\lf^{\ast})'(x))$. Hence, using \eqref{inf2} we get
\begin{align*}
z-x & =\frac{z \lf((\lf^\ast)'(x))}{\lf^\ast(x)}\left(\frac{x(\lf^\ast)'(x)-\lf((\lf^\ast)'(x))}{\lf((\lf^\ast)'(x))}-\frac{x z(\lf_B^\ast)'(z(1-\varepsilon_1))}{z \lf((\lf^\ast)'(x))} \right)\\
& =\frac{z \lf((\lf^\ast)'(x))}{\lf^\ast(x)}\left(\frac{x(\lf^\ast)'(x)-x(\lf_B^\ast)'(z(1-\varepsilon_1))}{ \lf((\lf^\ast)'(x))}-1 \right).
\end{align*}
Since $x\sim z$, \eqref{psipl} implies that
$$\frac{x(\lf^\ast)'(x)-x(\lf^\ast)'(\frac{1-\varepsilon_1}{B}z)}{ \lf((\lf^\ast)'(x))}\to\ln\frac{B}{1-\varepsilon_1}.$$
By \eqref{eq4p}, we have
$$I(z)\leq  \varepsilon_1 z(\lf_B^\ast)'(z) \exp\left\{\frac{z \lf((\lf^\ast)'(x))}{\lf^\ast(x)}\left(\ln\frac{B}{1-\varepsilon_1}-1+o(1)\right)\right\}+o(1).$$
Since $\frac{z \lf((\lf^\ast)'(x))}{\lf^\ast(x)}\sim\frac{\lf((\lf^\ast)'(x))}{(\lf^\ast)'(x)} \to\infty$, we see that if $\ln(B/(1-\varepsilon_1))-1<0$, then $I(z)\to0$. This happens if $B<(1-\varepsilon_1)e$. 

Passing to the limit as $\varepsilon_1\to0$ and $B\to e$, we obtain
\begin{align*}
\limsup_{x\to\infty}\frac{\ln \P(R>x)}{\lf(x)}\leq-e
\end{align*}
Since $\lf(x)\sim e^{-\varepsilon}f(x)$, we get 
\begin{align*}
\limsup_{x\to\infty}\frac{\ln \P(R>x)}{f(x)}\leq-e^{1-\varepsilon}.
\end{align*}
$\varepsilon\to0$ gives the result.

\subsection*{Case III}
Function $x\mapsto -\ln\overline{F}_M(1-\frac1x)$ is slowly varying (still, nondecreasing) and we need a different approach. The one we will present here has been inspired by \citep{HitWes09}, where a similar technique was used to solve the case $\ln\P(M>1-1/x)=-\beta(\ln x)^\eta$. 
We will prove that for $0<B<1$
\begin{align*}
I(z)=\E \exp\{\lf_B^\ast(zM)-\lf_B^\ast(z)+z\}
\end{align*}
converges to $0$ as $z\to\infty$, where $\lf_B^\ast(x):=B\lf^\ast(x/B)$ and $\lf\in SR(1)$ is as in Case I. The function $f$ is assumed to be ultimately strictly convex and this implies $\lf$ is also ultimately strictly convex.

We have for $\alpha\in(0,1)$,
\begin{align*}
I(z) & \leq e^z \int_{1-\alpha}^1 dF_M(x)+\exp\left\{z+\lf_B^\ast(z(1-\alpha))-\lf_B^\ast(z)\right\}\int_0^{1-\alpha} dF_M(x)\\
& \leq e^z \overline{F}_M(1-\alpha)+\exp\left\{z+\lf_B^\ast(z(1-\alpha))-\lf_B^\ast(z)\right\}\\
&\leq \exp\left\{z-\alpha \lf(\frac1\alpha)\right\}+\exp\left\{z-z\alpha(\lf_B^\ast)'(z(1-\alpha))\right\},
\end{align*}
where the last inequality follows from $\overline{F}_M(1-\alpha)=\exp\{-\alpha f(\frac1\alpha)\}\leq\exp\{-\alpha \lf(\frac1\alpha)\}$.

Define the function $\alpha=\alpha(z)$ by the identity $z-\alpha \lf(1/\alpha)=-\rho$, where $\rho>0$ is fixed. Obviously, $\alpha(z)\to0$ as $z\to\infty$. We will show that the second term above goes to $0$ for some $B>0$.
We are to show that
\begin{align}\label{B0}
\lim_{z\to\infty}\{1-\alpha( \lf_B^\ast)'(z(1-\alpha))\}<0.
\end{align}
This will happen if for $z$ large enough
\begin{align*}
\frac{1}{\alpha}<(\lf_B^\ast)'(z(1-\alpha)),
\intertext{which is equivalent to}
\lf'\left(\frac{1}{\alpha}\right)<\frac{z(1-\alpha)}{B}.
\end{align*}
By the definition of $\alpha(z)$ and Monotone Density Theorem $z\sim \alpha \lf(1/\alpha)\sim \lf'(1/\alpha)$, thus, if
$B<1$, \eqref{B0} holds true for large $z$ and we obtain
$$\lim_{z\to\infty} I(z)\leq e^{-\rho}\to 0,\qquad\mbox{ as }\rho\to\infty$$
and so (after taking $B\to1$)
\begin{align*}
\limsup_{x\to\infty}\frac{\ln \P(R>x)}{\lf^{\ast\ast}(x)}\leq-1.
\end{align*}
$f$ is ultimately strictly convex, thus $\lf$ is also ultimately strictly convex. This implies that $\lf^{\ast\ast}(x)=\lf(x)$ for large $x$ and since $\lf(x)\sim f(x)$, we finally obtain that
\begin{align*}\limsup_{x\to\infty}\frac{\ln \P(R>x)}{f(x)}\leq-1.
\end{align*}

\section{Lower bound}\label{lower}
We start with \eqref{it}, in which $(M_n)_n$ are i.i.d. copies of $M$ and $R_0=1$.
Again, without loss of generality we may assume that $Q=q=1$. 
Define the sequence
$$x_n=1+(1-\delta_{n})x_{n-1},\qquad n\geq 1,$$ 
with $x_0=1$ and $\delta_n\in(0,1)$ for each $n$. Then $\P(R_1>x_1)=\P(M_1>1-\delta_1)$ and
\begin{align*}
\P(R_n> x_n)&=\P\left(M_n R_{n-1}>(1-\delta_n)x_{n-1}\right)\\
&\geq \P(M_n> 1-\delta_n)\P(R_{n-1}>x_{n-1})\\
&\geq \prod_{i=1}^{n} \P(M>1-\delta_i).
\end{align*}
Since 
$$R\stackrel{d}{=}\sum_{j=1}^\infty \prod_{k=1}^{j-1} M_k\geq \sum_{j=1}^{n+1} \prod_{k=1}^{j-1} M_k\stackrel{d}{=}R_n,$$
we get for any $n\geq 1$,
$$\P(R>x_n)\geq \prod_{i=1}^n \P(M>1-\delta_i).$$
In each of the considered cases there exists a strictly convex differentiable function $h$ such that $f(x)\sim h(x)$. 
We take the logarithms and divide both sides by $h(x_{n})$ to obtain
\begin{align*}
\frac{\ln\P(R>x_n)}{h(x_n)}\geq \frac{\sum_{i=1}^n \ln\P(M>1-\delta_i)}{h(x_{n})}.
\end{align*}
Our aim is to choose a sequence $(\delta_n)_n$ in such way that the right-hand side of the above inequality tends to the optimal constant, which was obtained in the previous section. 

If $x_n\to\infty$ and $x_n$ is strictly monotone, we may use Stoltz--Ces\`aro theorem (see \cite[Problem 70]{PS98}) to get
$$I=\lim_{n\to\infty} \frac{\sum_{i=1}^n \ln\P(M>1-\delta_i)}{h(x_n)}=\lim_{n\to\infty} \frac{\ln\P(M>1-\delta_n)}{h(x_n)-h(x_{n-1})}.$$
Recall that $\ln\P(M>1-\delta_n)=-\delta_n f(1/\delta_n)$ and that by convexity of $h$ we have $h(x_{n})-h(x_{n-1})\geq h'(x_{n-1})(x_n-x_{n-1})$. So, 
\begin{align}\label{xn}
I\geq \lim_{n\to\infty}\frac{-\delta_n f(\frac{1}{\delta_n})}{h'(x_{n-1})(x_n-x_{n-1})}.
\end{align}

We distinguish two cases:
\begin{description}
\item[Case I] $f\in R(r^\ast)$, $r^\ast\geq1$,
\item[Case II] $f\in \Gamma$.
\end{description}

\subsection*{Case I}
$f\in R(r^\ast)$, $r^\ast\geq 1$. For $a\in(0,1)$, take $\delta_n=(1-a)/x_{n-1}$. Then $x_n=an+1$.
For such a choice by \eqref{xn}, we obtain
\begin{align*}
I & \geq \lim_{n\to\infty}\frac{-\frac{1-a}{x_{n-1}}f(\frac{x_{n-1}}{1-a})}{h'(x_{n-1})a}=-\frac{(1-a)^{1-r^\ast}}{ar^\ast}=:-i(a).
\end{align*}
The infimum of $i(a)$ is attained at the point $a=(r^{\ast})^{-1}$. Finally,
$$I\geq -r^{r^\ast-1}.$$
Particularly, if $r^\ast=1$, then $I\geq-1$.

\subsection*{Case II}
Since $f(x)\sim h(x)$ and both $f$ and $h$ belong to $\Gamma$, they have common auxiliary function $g=h/h'$. Take $x_1=1+\varepsilon$ and $1/\delta_n=x_{n-1}+g(x_{n-1})$ for $n>1$. 
We have
$$x_n-x_{n-1}=1-\delta_n x_{n-1}=\frac{g(x_{n-1})}{x_{n-1}+g(x_{n-1})}.$$
For such a choice by \eqref{xn}, we obtain
\begin{align*}
I\geq-\lim_{n\to\infty} \frac{f\left(x_{n-1}+g(x_{n-1})\right)}{h(x_{n-1})}.
\end{align*}
If $x_{n}\to\infty$, then by \eqref{Gamma} we get $I\geq -e$. 
Assume that $x_n\not\to \infty$. Thus, $x_n$ being increasing, converges to some $p>x_1=1+\varepsilon$, say. Then we have $0=\frac{g(p)}{p+g(p)}$ - a contradiction, since $g(x)>0$ for any $x>1$.

Let us get back to the general situation. From $\liminf_{n\to\infty} \frac{\ln \P(R>x_n)}{f(x_n)}\geq-B$, we need to conclude that
\begin{align}\label{liminf}
\liminf_{x\to\infty} \frac{\ln \P(R>x)}{f(x)}\geq -B.
\end{align}
Observe first that if $x_{n-1}< x\leq x_{n}$, then $\frac{\ln \P(R>x)}{f(x)}\geq \frac{\ln \P(R>x_n)}{f(x_{n-1})}$. Therefore,
$$\liminf_{x\to\infty} \frac{\ln \P(R>x)}{f(x)}\geq \liminf_{n\to\infty} \frac{f(x_n)}{f(x_{n-1})}\frac{\ln \P(R>x_n)}{f(x_n)}.$$
It is left to show that
$$\frac{f(x_n)}{f(x_{n-1})}\to 1.$$
But $x_n-x_{n-1}\in(0,1)$, so the only case that needs some more attention is when $f\in\Gamma$. In this case, we have
$$\frac{f(x_n)}{f(x_{n-1})}=\frac{f(x_{n-1}+\frac{1}{x_{n-1}+g(x_{n-1})}g(x_{n-1}))}{f(x_{n-1})}\to1,$$
since the convergence in \eqref{Gamma} is uniform (see \citep[Proposition 3.10.2]{BGT89}).
Thus, \eqref{liminf} holds and we finally obtain
$$-B\leq \liminf_{x\to\infty} \frac{\ln \P(R>x)}{f(x)}\leq \limsup_{x\to\infty} \frac{\ln \P(R>x)}{f(x)}\leq -B.$$

\section{Concluding remarks}\label{remarks}
\begin{enumerate}
\item In Theorem \ref{mainr}, the assumption of ultimate strict convexity of $f$ when $f\in R(1)$ implies that $f$ is ultimately continuous and so $M$ may not have atoms in the left neighbourhood of $1$. Lower bound can be obtained without this assumption by the use of Goldie--Gr{\"u}bel inequality (see (5.7) in \citep{GG96} or Proposition 1 in \citep{HitWes09} for general formulation). The problem is with the upper bound as it is not true that $f(x)\sim f^{\ast\ast}(x)$ for $f\in R(1)$ in general. Possibly this case can be proved without the assumption of convexity using the approach of Hitczenko \citep{Hit10}.
\item
Consider Theorem \ref{mainr} and Corollary \ref{ato} with $\P(M\in[0,1])=1$ replaced by $\P(|M|\leq 1)=1$. It turns out that the same conclusion holds if additionally
\begin{align}\label{tr}
\ln \P\left(|M|>1-\frac 1x\right)\sim\ln \P\left(M>1-\frac 1x\right).
\end{align}
Indeed, following the idea of \citep{GG96},
$$\P(R>x)\leq \P(|R|>x)\leq \P(q\sum_{j=1}^\infty\prod_{k=1}^{j-1} |M_k|>x).$$
We therefore obtain that
$$\limsup_{x\to\infty}\frac{\ln\P(R>x)}{-\frac{x}{q}\ln \P(|M|>1-\frac {q}x)}\leq -B,$$
where $B$ is the optimal constant. On the other hand, the approach laid out in Section \ref{lower} for the lower bound holds also when $M$ is allowed to take negative values, thus
$$\liminf_{x\to\infty} \frac{\ln\P(R>x)}{-\frac{x}{q}\ln \P(M>1-\frac {q}x)}\geq -B$$
and if \eqref{tr} holds, then 
$$\lim_{x\to\infty}\frac{\ln\P(R>x)}{-\frac x{q}\ln \P(M>1-\frac {q}x)}= -B.$$

\item It is very interesting that we were able to show the asymptotic of log-tail of $R$ for $f\in\Gamma$, but not for $f\in R(\infty)\setminus\Gamma$.
The latter case is much harder since neither \eqref{Gamma} nor \eqref{gg} hold, which were crucial in the proof of Theorem \ref{mainr} when $f\in\Gamma$. 

If $f\in R(\infty)$ is nondecreasing (this is our case), then by \citep[Proposition 2.4.4]{BGT89} $f\in KR(\infty)$, that is, $f$ has the following representation:
$$f(x)=\exp\left\{\eta(x)+z(x)+\int_0^x\frac{\varepsilon(t)}{t}\,dt\right\},$$
where $\eta(x)\to d$ and $\varepsilon(x)\to\infty$ as $x\to\infty$ and $z(\cdot)$ is nondecreasing, but unfortunately this seems of little help.

One of the examples of $f\in R(\infty)\setminus\Gamma$ is $f(x)=\exp(2x-\cos x)$, which is even strictly convex. In this case, we have
$$\frac{f\left(x+ug(x)\right)}{f(x)}=\exp\left\{\cos x-\cos\left(x+\frac{u}{2+\sin x}\right)+\frac{2 u}{2+\sin(x)}\right\},$$
which is periodic, thus \eqref{Gamma} does not hold.
The best we can get in this situation is 
$$-e^3\leq \liminf_{x\to\infty} \frac{\ln \P(R>x)}{f(x)}\leq \limsup_{x\to\infty} \frac{\ln \P(R>x)}{f(x)}\leq -e^{-1},$$
using $e^{2x-1}\leq f(x)\leq e^{2x+1}$. This is much stronger than \eqref{genhit}, but still not fully satisfactory.
\end{enumerate}

\section*{Acknowledgements}
The author would like to thank Professor Jacek Weso{\l}owski, Kamil Szpojankowski and Marcin {\'S}wieca for valuable discussions. The author would also like to thank Wojciech Matysiak and Kamil Kosi{\'n}ski for their helpful comments. 
The author is grateful to the referees for their careful reading of the manuscript, which helped to remove many typos and inaccuracies.  The proof of Lemma \ref{lem1} was shortened thanks to a referee's remark.

\bibliographystyle{imsart-nameyear}


\def\polhk#1{\setbox0=\hbox{#1}{\ooalign{\hidewidth
  \lower1.5ex\hbox{`}\hidewidth\crcr\unhbox0}}}

\end{document}